\documentclass{amsart}

\usepackage{amsmath,amssymb,cite}
\usepackage{amsthm}
\usepackage{amsrefs}
\usepackage{mathtools}

\usepackage[plainpages=false,pdfborder={0 0 0}]{hyperref}

\newtheorem{thm}{Theorem}[section]
\newtheorem{lem}[thm]{Lemma}
\newtheorem{prop}[thm]{Proposition}

\newtheorem{cor}[thm]{Corollary}
\theoremstyle{definition}

\theoremstyle{remark}
\newtheorem{rmk}[thm]{Remark}

\numberwithin{equation}{section}

\begin{document}

\title[Prescribed Mean Curvature]{Conformal scalar-flat metrics with prescribed boundary mean curvature}

\author{Jiashu Shen}
\address{School of Mathematical Sciences, Zhejiang University, Hangzhou, Zhejiang Province 310058, China}
\curraddr{Changshu City High School, Suzhou, Jiangsu Province 215500, China}
\email{shenjiashu0529@gmail.com}

\author{Hongyi Sheng}
\address{Institute for Theoretical Sciences, Westlake Institute for Advanced Study, Westlake University, Hangzhou, Zhejiang Province 310024, China}
\email{shenghongyi@westlake.edu.cn}





\begin{abstract}
Let $(M, g)$ be a compact Riemannian manifold with boundary $\partial M$. Given a function $f$ on $\partial M$, we consider the problem of finding a conformal metric of $g$ with zero scalar curvature in $M$ and prescribed mean curvature $f$ on $\partial M$. Through the construction of local test functions, we resolve most of the remaining open cases from Escobar's work \cite{article15} and establish new solvability conditions.
\end{abstract}

\maketitle



\section{Introduction}\label{sec1}
The Yamabe problem, solved by Trudinger \cite{article24}, Aubin \cite{article1}, and Schoen \cite{article21}, asserts that any Riemannian metric on a closed manifold is conformal to a metric with constant scalar curvature. There are two natural ways to extend the Yamabe problem. One extends the result to manifolds with boundary, which leads to the study of the boundary Yamabe problem. The other is to consider whether there exists a conformal metric such that the scalar curvature is a given function, which is known as the prescribed curvature problem.

For the Yamabe problem on manifolds with boundary, due to the introduction of the mean curvature on the boundary, it can be further divided into two cases. The first case is to find a conformal metric such that the scalar curvature is constant and the mean curvature is zero (the minimal boundary case); the second case is to find a conformal metric such that the scalar curvature is zero and the mean curvature is constant (the scalar-flat case). Both cases were solved by the works of Escobar \cite{article13, article14}, Brendle-Chen \cite{article4}, and Chen \cite{article5}. Escobar \cite{article13, article14} primarily studied the cases where $\partial M$ is either non-umbilic or totally umbilic with $M$ locally conformally flat, while Brendle-Chen \cite{article4} and Chen \cite{article5} introduced new test functions to solve the case where $M$ is not locally conformally flat, and reduced the remaining cases
to the positive mass theorem. Naturally, we can also study the prescribed curvature problems corresponding to these two Yamabe problems on manifolds with boundary.

More precisely, let $(M,g)$ be a compact Riemannian manifold of dimension $n\geqslant3$ with boundary $\partial M$. We use $R_{g}$ to denote the scalar curvature of $M$ and $h_{g}$ the mean curvature of $\partial M$.
Then we have the following two prescribed curvature problems:
\begin{enumerate}
	\item[(1)] For a given function $f$ on $M$, find a metric $\tilde{g}$ in the conformal class of $g$ such that $R_{\tilde{g}}=f$ and $h_{\tilde{g}}=0$.
	
	\item[(2)] For a given function $f$ on $\partial M$, find a metric $\tilde{g}$ in the conformal class of $g$ such that $R_{\tilde{g}}=0$ and $h_{\tilde{g}}=f$.
\end{enumerate}

In this paper, we focus on problem (2) and study the solvability conditions on $(M,g)$ and $f$. Let $\tilde{g}=u^{\frac{4}{n-2}}g$. Then problem (2) is equivalent to the following equations:
\begin{equation}\label{conformal eqn}
	\begin{cases}\Delta_g u-\frac{n-2}{4(n-1)} R_{g}u=0, & x\in M, \\ 
 u_\nu+\frac{n-2}{2} h_{g} u=\frac{n-2}{2} f u^{\frac{n}{n-2}},& x\in\partial M.\end{cases}
\end{equation}

For $\phi\in H^{1}(\overline{M})$, define the following energy and the Yamabe functional:
$$
\begin{aligned}
    E_g(\phi)&=\int_M\left(\frac{4(n-1)}{n-2}\left|\nabla\phi\right|_{g}^2+R_g \phi^2\right) d V_g+\int_{\partial M} 2 h_g \phi^2 d \sigma_g,\\
    \mathcal{Q}_{g}(\phi)&=\frac{E_{g}(\phi)}{\left(\int_{\partial M} \phi^{\frac{2(n-1)}{n-2}} d \sigma_g\right)^{\frac{n-2}{n-1}}},
\end{aligned}
$$
and define the Sobolev quotient by
$$
\mathcal{Q}(M, \partial M)=\inf _{0<\phi \in H^{1}(M)} \mathcal{Q}_g(\phi).
$$
Depending on whether the Sobolev quotient is negative, zero, or positive, there are different ways to deal with this problem.

When $\mathcal{Q}(M, \partial M)=0$, Escobar \cite{article15} provided a necessary and sufficient condition for the existence of a solution, using the method of upper and
lower solutions: the problem has a solution if and only if $f$ changes sign and $\int_{\partial M}f d \sigma_g<0.$

When $\mathcal{Q}(M, \partial M)<0$, Escobar \cite{article15} proved that if $f<0$, the problem has a solution. 

When $\mathcal{Q}(M, \partial M)>0$, Escobar \cite{article15} used the variational method and studied the variation of energy $E_g(\phi)$ on the constraint set
$$C_{\frac{n}{n-2},f}=\{  \phi\in H^{1}(\overline{M}):\int_{\partial M} f\left|\phi\right|^{\frac{n}{n-2}+1}d\sigma_g=1\}.$$
Let $E(\frac{n}{n-2},f)$ denote the infimum of $E_g(\phi)$ taken over $C_{\frac{n}{n-2},f}$. He proved the following criterion theorem.

\begin{thm}[Escobar \cite{article15}]\label{criterion}
	Let $(M, g)$ be a compact Riemannian manifold with boundary, with dimension $n \geqslant 3$. If $f$ satisfies 
\begin{equation}\label{criterion ineq}
    \left(\max _{x \in \partial M} f(x)\right)^{\frac{n-2}{n-1}} E\left(\frac{n}{n-2}, f\right)<\mathcal{Q}(B^{n},\partial B^{n}),
\end{equation}
then there is a smooth positive solution to (\ref{conformal eqn}). Here $B^{n}$ is the unit $n$-ball.
\end{thm}

In the following cases when $(M, g)$ has positive Sobolev quotient, Escobar \cite{article15} proved the inequality (\ref{criterion ineq}) is satisfied, and thus by the above criterion, problem (\ref{conformal eqn}) has a smooth positive solution.

\begin{enumerate}	
	\item[(1)]$n=3$, $M$ is not conformally equivalent to the ball, with any function $f$ which is somewhere positive.
	
	\item[(2)]$n=4,5$, $\partial M$ is umbilic and $M$ is not conformally equivalent to the ball, with $f$ satisfying condition (\ref{*}).
	
	\item [(3)]$n\geqslant 6$, $\partial M$ is umbilic, $M$ is locally conformally flat, and $M$ is not conformally equivalent to the ball, with $f$ satisfying condition (\ref{*}).

     \item [(4)]$n\geqslant 6$, $\partial M$ is not umbilic, with $f$ satisfying condition (\ref{**}).
\end{enumerate}

Here condition (\ref{*}) on $f$ means:
\begin{equation}\label{*}
    \begin{cases}
        f \text{ is somewhere positive and achieves a global maximum}\\
        \text{at }p\in\partial M\text{ where }\nabla^k f(p)=0\text{ for }k=1, \cdots, n-2.
    \end{cases}
\end{equation}

\begin{rmk}
    Note that condition (\ref{*}) is automatically satisfied for any smooth function (which is somewhere positive) defined on the boundary on a $3$-dimensional manifold.
\end{rmk}

Meanwhile, condition (\ref{**}) on $f$ means:
\begin{equation}\label{**}
    \begin{cases}
        f \text{ is somewhere positive and achieves a global}\\
        \text{maximum at a non-umbilic point }p\in\partial M\text{ where }\\
        \Delta_{\partial M} f(p)\leqslant c(n)\|\pi_{g}-h_{g} g\|^2(p).
    \end{cases}
\end{equation}
Here $\pi_{g}$ is the second fundamental form of $\partial M$, and $c(n)$ is a suitable positive constant that depends only on the dimension of $M$.

\bigskip
Inspired by the ideas in \cite{article27, article4, article5}, we handle the remaining cases of this problem when $\mathcal{Q}(M, \partial M)>0$, namely

\begin{itemize}
    \item $n=4,5$, and $\partial M$ is not umbilic;
    \item $n\geqslant 6$, $\partial M$ is umbilic and $M$ is not locally conformally flat.
\end{itemize}

First, for $n\geqslant 5$ and non-umbilic $\partial M$, we have the following result.

\begin{thm}\label{n=5, nonumbilic}
Let $(M, g)$ be a compact Riemannian manifold with boundary, with dimension $n \geqslant 5$, and $\mathcal{Q}(M,\partial M)>0$. If $\partial M$ is not umbilic and $f$ satisfies condition (\ref{**}), then problem (\ref{conformal eqn}) has a smooth positive solution.
\end{thm}

For $4$-dimensional Riemannian manifolds with non-umbilic boundary, we can solve problem (\ref{conformal eqn}) for a wider range of functions $f$.

\begin{thm}\label{n=4, nonumbilic}
Let $(M^4, g)$ be a $4$-dimensional compact Riemannian manifold with boundary, and $\mathcal{Q}(M,\partial M)>0$. If $\partial M$ is not umbilic, and $f$ is somewhere positive and achieves a global maximum at a non-umbilic point, then problem (\ref{conformal eqn}) has a smooth positive solution.
\end{thm}

Then for $n\geqslant6$ with umbilical boundary $\partial M$, we follow Brendle–Chen \cite{article4} and employ the positive mass theorem in higher dimensions (see, for example, \cite{article23}) to remove the requirement that the manifold be locally conformally flat.

\begin{thm}\label{umbilic}
Let $(M, g)$ be a compact Riemannian manifold with boundary, with dimension $n \geqslant 6$, and $\mathcal{Q}(M,\partial M)>0$. If $\partial M$ is umbilic, $M$ is not conformally equivalent to the ball, and $f$ satisfies condition (\ref{*}), then problem (\ref{conformal eqn}) has a smooth positive solution.
\end{thm}

Combined with the above existence results of Escobar \cite{article15}, we can now solve problem (\ref{conformal eqn}) under the following conditions.

\begin{cor}
Let $(M, g)$ be a compact Riemannian manifold with boundary, with dimension $n \geqslant 3$, and $\mathcal{Q}(M,\partial M)>0$. Then problem (\ref{conformal eqn}) has a smooth positive solution, provided that one of the following conditions is satisfied.

    \begin{enumerate}	
    \item[(1)]$n=3$, $M$ is not conformally equivalent to the ball, with any function $f$ which is somewhere positive.
    
    \item [(2)]$n=4$, $\partial M$ is not umbilic, with any function $f$ which is somewhere positive and achieves a global maximum at a non-umbilic point.

     \item [(3)]$n\geqslant 5$, $\partial M$ is not umbilic, with $f$ satisfying condition (\ref{**}).
	
	\item [(4)]$n\geqslant 4$, $\partial M$ is umbilic, $M$ is not conformally equivalent to the ball, with $f$ satisfying condition (\ref{*}).
\end{enumerate}
\end{cor}

Therefore, we have resolved most of the remaining cases left in \cite{article15}.

Note when $\partial M$ is umbilic, there is an additional assumption that $M$ is not conformally equivalent to the ball $B^n$. This is because the globally constructed test function makes use of the Green's function for the conformal Laplacian, and the validity of (\ref{criterion ineq}) depends on a certain parameter of the Green's function being strictly positive, which in turn relies on the positive mass theorem \cite{article23} for asymptotically flat manifolds. Indeed, this assumption can be slightly weakened. See Section \ref{sec5} for more details.

On the other hand, the prescribed mean curvature problem on $B^n$ is quite challenging. However, there are several results in this direction; see, for example, Escobar \cite{article15}, Chang-Xu-Yang \cite{article7}, Xu-Zhang \cite{article25}, Escobar-Garcia \cite{article17}, and Ahmedou-Djadli-Malchiodi \cite{article10}.

The paper is organized as follows. In Section \ref{sec2}, we define basic notation and introduce some algebraic preliminaries. In Section \ref{sec3}, we construct test functions and derive useful estimates. In Section \ref{sec4}, we focus on $n=4,5$, and manifolds with non-umbilic boundary, where we prove Theorem \ref{n=5, nonumbilic} and Theorem \ref{n=4, nonumbilic}. In Section \ref{sec5}, we discuss the case of $n\geqslant 6$ and manifolds with umbilic boundary, where we prove Theorem \ref{umbilic}. In particular, we make use of the positive mass theorem \cite{article23} to remove the
assumption in \cite{article15} that the manifold has to be locally conformally flat.

\section{Preliminaries}\label{sec2}
Let $(M, g_0)$ be a compact Riemannian manifold with boundary, with dimension $n \geqslant 3$. For a Riemannian metric $g$, $\nu$ will denote the outward unit normal vector to $\partial M$ with respect to $g$, $\pi_{g}$ the second fundamental form of $\partial M$, and $\Delta_{\partial M}$ the Laplace-Beltrami operator on $\partial M$. Furthermore, $\Delta_{g}$, $R_{g}$, and $h_{g}$ will denote the Laplace-Beltrami operator on $M$, the scalar curvature of $M$, and the mean curvature of $\partial M$, respectively. We use $dV_{g}$ and $d\sigma_{g}$ to denote the volume form of $M$ and $\partial M$, and use $dx$ and $d\sigma$ to denote the volume form of $\mathbb{R}_+^{n}$ and $\partial\mathbb{R}_+^{n}$.

We use $B_{\delta}(x)$ to denote the ball in $\mathbb{R}^{n}$ of radius $\delta$ with center $x\in\mathbb{R}^{n}$ and  $\Omega_{\delta}(p)$ the coordinate ball in $M$ of radius $\delta$ with center $p\in M$. We use $\omega_{n-1}$ to denote the surface area of the Euclidean unit sphere $\mathbb{S}^{n-1}\subset\mathbb{R}^{n}$.

Since the Sobolev quotient $\mathcal{Q}(M,\partial M)>0$, we may assume $R_{g_0}>0$ and $h_{g_0}=0$ up to a conformal change of the metric \cite{article14}*{Lemma 1.1}. 

Fix $p\in\partial M$, and let $(x_1, \ldots, x_n)$ denote the Fermi coordinates around $p$. (Note here we do not assume $p$ is an umbilic point of $\partial M$.) More specifically, for small $\delta>0$ there exists a smooth map $\varphi:B_{\delta}(0)\cap \mathbb{R}^{n}_{+} \to M$ such that $\varphi(0)=p$. We blur the distinction between $f(\varphi(x))$ and $f(x)$, and use $f(x)$ to denote both of them to simplify our notation. There are some important properties of the metric $g$ in these coordinates, where $i, k$ run from $1$ to $n$ and $a,b,c$ run from $1$ to $n-1$:

\begin{equation*}
    \begin{cases}
        g_{ik}(0)=\delta_{ik},&\\
        g_{in}(x)=\delta_{in},& \text{for }x\in B_{\delta}(0)\cap \mathbb{R}^{n}_{+},\\
        \partial_{a}g_{bc}(0)=0,&\\
        \sum_{a=1}^{n-1}x_{a}g_{ab}(x)=x_b,& \text{for }x\in B_{\delta}(0)\cap \partial \mathbb{R}^{n}_{+}.\\
    \end{cases}
\end{equation*}

If we write $g=\exp(h)$ where exp denotes
the matrix exponential, then $h$ is a symmetric $2$-tensor and satisfies the following properties \cite{article27}:

\begin{equation*}
    \begin{cases}
        h_{ik}(0)=0,&\\
        h_{in}(x)=0,& \text{for }x\in B_{\delta}(0)\cap \mathbb{R}^{n}_{+},\\
        \partial_{a}h_{bc}(0)=0,&\\
        \sum_{a=1}^{n-1}x_{a}h_{ab}(x)=0,& \text{for }x\in B_{\delta}(0)\cap \partial \mathbb{R}^{n}_{+}.\\
    \end{cases}
\end{equation*}

Recall that we assumed $R_{g_0}>0$ and $h_{g_0}=0$. Then by Marques \cite{article19}*{Proposition 3.1}, there exists a conformal metric $g=u^{\frac4{n-2}}g_0$ with 
$$u(p)=1, \text{  and  }\operatorname{det}g(x)=1+O(|x|^{2d+2})\,\,\text{ near }p$$ 
where $d=\left[\frac{n-2}{2}\right]$. 
This determinant condition implies that the mean curvature of $\partial M$ with respect to $g$ satisfies
$$h_{g}(x)=-\frac{1}{2} g^{ab} \partial_n g_{ab}(x)=-\frac{1}{2} \partial_n\left(\log \operatorname{det}g\right)(x)=O\left(|x|^{2 d+1}\right).$$
Moreover, if $g=\exp(h)$, then $\operatorname{tr}(h)(x) = O(|x|^{2d+2})$.

Let
$$H_{ik}(x)\triangleq \sum_{|\alpha|=1}^{d}h_{ik,\alpha}x^{\alpha}$$
denote the Taylor expansion of order $d$ associated with the function $h_{ik}(x)$, where $\alpha$ is a multi-index. So $h_{ik}(x)$ can be written as $h_{ik}(x) =
H_{ik}(x) + O\left(|x|^{d+1}\right)$. Then $H$ is a symmetric trace-free 2-tensor on $\mathbb{R}^{n}_{+}$ satisfying
\begin{equation*}
    \begin{cases}
        H_{ik}(0)=0,\\
        \partial_{a}H_{bc}(0)=0,\\
        H_{in}(x)=0,& x\in\mathbb{R}^{n}_{+},\\
        \sum_{a=1}^{n-1}x_{a}H_{ab}(x)=0,& x\in\partial \mathbb{R}^{n}_{+}.\\
    \end{cases}
\end{equation*}

Consider the case where $|\alpha|=1$. For $\alpha=(0,0,\cdots,1,\cdots,0)$ with $1$ in the $j$-th position, $h_{ik,\alpha}=\partial_{j}h_{ik}(0)=\partial_{j}g_{ik}(0).$ So 
$$
h_{ik,\alpha}= \begin{cases}0, & \text { if } k=n \text { or }j\neq n, \\ -{2}(\pi_{g})_{ik}(0), & \text { if } i, k=1, \ldots,n-1\text { and } j=n.\end{cases}
$$
Thus, we have
\begin{equation}\label{h=pi}
\sum_{|\alpha|=1}\sum_{i,k=1}^{n}|h_{ik,\alpha}|^{2}={4}\|\pi_{g}(0)\|^{2}.
\end{equation}

We then define the algebraic Schouten tensor and Weyl tensor as in \cite{article3}:
$$
\begin{aligned}
A_{i j} & =\partial_i \partial_m H_{m j}+\partial_m \partial_j H_{i m}-\Delta_g H_{i j}-\frac{1}{n-1} \partial_m \partial_p H_{m p} \delta_{i j},\\
    Z_{i j k l} & =\partial_i \partial_k H_{j l}-\partial_i \partial_l H_{j k}-\partial_j \partial_k H_{i l}+\partial_j \partial_l H_{i k}\\
    & \quad +\frac{1}{n-2}\left(A_{j l} \delta_{i k}-A_{j k} \delta_{i l}-A_{i l} \delta_{j k}+A_{i k} \delta_{j l}\right) .
\end{aligned}
$$

\begin{prop}[\cite{article27, article4}]\label{vanishing  H}
    If the tensor $H$ satisfies
    $$
\begin{cases}Z_{ijkl}=0, & \text { in } \mathbb{R}_{+}^n, \\ \partial_n H_{i j}=0, & \text { on } \partial \mathbb{R}_{+}^n,\end{cases}
$$
then $H=0$ in $ \mathbb{R}_{+}^n.$
\end{prop}

\section{Test functions}\label{sec3}
In this section, we follow Brendle-Chen \cite{article4}, Chen \cite{article5}, and Almaraz \cite{article27} to construct test functions. 

For $\varepsilon>0$, let $$v_\varepsilon(x)=\varepsilon^{\frac{n-2}{2}}\left(\left(\varepsilon+x_n\right)^2+\sum_{a=1}^{n-1} x_a^2\right)^{-\frac{n-2}{2}},\quad x \in \mathbb{R}_{+}^n .$$ 
The function $v_{\varepsilon}$ is the standard bubble, and it satisfies \cite{article5}

\begin{equation}
\begin{cases}
\Delta_g v_\varepsilon=0, & \text { in } \mathbb{R}_{+}^n, \\ 
\partial_n v_\varepsilon+(n-2) v_\varepsilon^{\frac{n}{n-2}}=0, & \text { on } \partial \mathbb{R}_{+}^n, \\
4(n-1)\left(\int_{\partial \mathbb{R}_{+}^n} v_\varepsilon(x)^{\frac{2(n-1)}{n-2}} d x\right)^{\frac{1}{n-1}}=Q\left(B^n, \partial B^{n}\right) .\\
\end{cases}
\end{equation}
Moreover, $v_{\varepsilon}$ satisfies the following basic estimates \cite{article5}.
\begin{equation}\label{ve bound}
\begin{cases}
\varepsilon^{\frac{n-2}{2}}(\varepsilon+|x|)^{-n+2} \leqslant  v_\varepsilon(x) \leqslant  C(n) \varepsilon^{\frac{n-2}{2}}(\varepsilon+|x|)^{-n+2},& x \in \mathbb{R}_{+}^n,\\
\left|\partial v_\varepsilon\right|(x) \leqslant  C(n) \varepsilon^{\frac{n-2}{2}}(\varepsilon+|x|)^{-n+1},& x \in \mathbb{R}_{+}^n.\\
\end{cases}
\end{equation}

Let $\eta:\mathbb{R}\to [0,1]$ be a smooth cut-off function with $|\eta'|\leqslant C$ and
$$
\eta\big|_{[0,4/3]}\equiv 1,\qquad\eta\big|_{[5/3,\infty)}\equiv 0.
$$ 
We define $\eta_{\delta}(x)=\eta(\frac{|x|}{\delta})$ for $x\in \mathbb{R}^{n}_{+}.$

By \cite{article27, article4}, there exists a smooth vector field $V$ on $\mathbb{R}_{+}^n$, which satisfies
\begin{equation}
\begin{cases}\sum_{k=1}^n \partial_k\left[v_\varepsilon^{\frac{2 n}{n-2}}\left(\eta_\delta H_{i k}-\partial_i V_k-\partial_k V_i+\frac{2}{n} \operatorname{div} V \delta_{i k}\right)\right]=0, & x\in \mathbb{R}_{+}^n \\ \partial_n V_a = V_n =0, & x\in \partial \mathbb{R}_{+}^n,\end{cases}
\end{equation}
for $i=1, \cdots, n$ and $a=1, \cdots, n-1$.
Furthermore, $V$ satisfies the following estimate
\begin{equation}\label{V property}
    \left|\partial^\beta V(x)\right| \leqslant C(n,|\beta|) \sum_{i, k=1}^n \sum_{1 \leqslant|\alpha| \leqslant d}\left|h_{i k, \alpha}\right|(\varepsilon+|x|)^{|\alpha|+1-|\beta|}
\end{equation}
for any multi-index $\beta$.

Then we define
\begin{align}
S_{i k} & =\partial_i V_k+\partial_k V_i-\frac{2}{n} \operatorname{div} V \delta_{i k},\nonumber\\
T_{i k} & =H_{i k}-S_{i k},\nonumber\\
\psi & = \sum^n_{l=1}\partial_{l}v_{\varepsilon}V_{l} + \frac{n-2}{2n}v_{\varepsilon}\operatorname{div}V.\label{psi def}
\end{align}
This $\psi$ serves as an appropriate perturbation of $v_\varepsilon$, and will be helpful for the construction of test functions.


Next, we note two propositions by Almaraz \cite{article27} concerning local estimates of $v_\varepsilon$ and $\psi$.
\begin{prop}[\cite{article27}*{Proposition 3.8}]\label{ve energy1}
If $\delta_{0}$ is sufficiently small, then
$$
\begin{aligned}
	& \quad\,\int_{B_\delta \cap \mathbb{R}_{+}^n}\left(\frac{4(n-1)}{n-2}\left|\nabla\left(v_\varepsilon+\psi\right)\right|_g^2+R_g\left(v_\varepsilon+\psi\right)^2\right) d x \\
	&\leqslant  4(n-1) \int_{B_\delta \cap \partial \mathbb{R}_{+}^n} v_\varepsilon^{\frac{2}{n-2}}\left(v_\varepsilon^2+2 v_\varepsilon \psi+\frac{n}{n-2} \psi^2-\frac{n-2}{8(n-1)^2} v_\varepsilon^2S_{n n}^2\right) d \sigma \\
	 &\quad+ \int_{\partial B_\delta \cap \mathbb{R}_{+}^n} \sum_{i=1}^n\left(\frac{4(n-1)}{n-2} v_\varepsilon \partial_i v_\varepsilon+\sum_{k=1}^n\left(v_\varepsilon^2 \partial_k h_{i k}-\partial_k (v_\varepsilon^2) h_{i k}\right)\right) \frac{x_i}{|x|} d \sigma\\
	&\quad- \frac{\theta}2 \sum_{i, k=1}^n \sum_{1 \leqslant|\alpha| \leqslant d}\left|h_{i k, \alpha}\right|^2 \varepsilon^{n-2} \int_{B_\delta \cap \mathbb{R}_{+}^n}(\varepsilon+|x|)^{2|\alpha|+2-2 n} d x \\
	&\quad+ C \varepsilon^{n-2} \sum_{i, k=1}^n \sum_{1 \leqslant|\alpha| \leqslant d}\left|h_{i k, \alpha}\right| \delta^{|\alpha|+2-n}+C \varepsilon^{n-2} \delta^{2 d+4-n}
\end{aligned}
$$
for all $ 0<2\varepsilon \leqslant \delta\leqslant\delta_{0}$. Here $\theta$ is a constant depending only on $(M, g_0)$.
\end{prop}

\begin{prop}[\cite{article27}*{(3.23)}]\label{ve energy2}
If $\delta_{0}$ is sufficiently small, then
$$
\begin{aligned}
		&\quad\, 4(n-1) \int_{B_\delta \cap \partial \mathbb{R}_{+}^n} v_\varepsilon^{\frac{2}{n-2}}\left(v_\varepsilon^2+2 v_\varepsilon \psi+\frac{n}{n-2} \psi^2-\frac{n-2}{8(n-1)^2} v_\varepsilon^2 S_{n n}^2\right) d \sigma \\
		& \leqslant \mathcal{Q}(B^{n}, \partial B^n)\left(\int_{B_\delta \cap \partial \mathbb{R}_{+}^n}\left(v_\varepsilon+\psi\right)^{\frac{2(n-1)}{n-2}} d \sigma\right)^{\frac{n-2}{n-1}}\\
  &\quad+C \sum_{i, k=1}^n \sum_{1 \leqslant|\alpha| \leqslant d}\left|h_{i k, \alpha}\right| \delta^{|\alpha|-n+1} \varepsilon^{n-1} \\
		& \quad+C \sum_{i, k=1}^n \sum_{1 \leqslant|\alpha| \leqslant d}\left|h_{i k, \alpha}\right|^2 \varepsilon^{n-1} \delta \int_{B_\delta \cap \partial \mathbb{R}_{+}^n}(\varepsilon+|x|)^{2|\alpha|-2 n+2} d \sigma
\end{aligned}
$$
for all $ 0<2\varepsilon \leqslant \delta\leqslant\delta_{0}$.
\end{prop}

In addition, we also have the following local estimate of $v_\varepsilon$ and $\psi$.
\begin{lem}\label{ve est}
If $\delta_{0}$ is sufficiently small, then
    $$\left|(v_{\varepsilon}+\psi)^{\frac{2(n-1)}{n-2}}-v_{\varepsilon}^{\frac{2(n-1)}{n-2}}\right|\leqslant C(n)\varepsilon^{n-1}(\varepsilon+|x|)^{3-2n}$$
for all $ 0<2\varepsilon \leqslant \delta\leqslant\delta_{0}$ and $|x|<\delta$.
\end{lem}

\begin{proof}
By Taylor's Theorem, there exists a constant $C(n)$ such that
    $$\left|(1+y)^{\frac{2(n-1)}{n-2}}-1\right|\leqslant C(n)|y|$$
for $|y|\leqslant \frac12$ sufficiently small. On the other hand, from the definition (\ref{psi def}) of $\psi$, the properties (\ref{V property}) of $V$ and (\ref{ve bound}), we have
\begin{align}\label{phi<ve}
		|\psi|& \leqslant  C\varepsilon^{\frac{n-2}{2}}(\varepsilon+|x|)^{1-n}\sum_{i, k=1}^{n}\sum_{|\alpha|=1}^{d}|h_{ik,\alpha}|(\varepsilon+|x|)^{|\alpha|+1}\nonumber\\
		&\quad+C\varepsilon^{\frac{n-2}{2}}(\varepsilon+|x|)^{2-n}\sum_{i, k=1}^{n}\sum_{|\alpha|=1}^{d}|h_{ik,\alpha}|(\varepsilon+|x|)^{|\alpha|}\nonumber\\
		&\leqslant  Cv_{\varepsilon}\sum_{|\alpha|=1}^{d}(\varepsilon+|x|)^{|\alpha|}\nonumber\\
  & \leqslant  Cv_{\varepsilon}(\varepsilon+|x|)\leqslant \frac{1}{2}v_{\varepsilon}
\end{align}
 for $\delta_0$ sufficiently small. Thus, we can take $y=\psi/v_\varepsilon$ and get
$$
\begin{aligned}
    \left|(v_{\varepsilon}+\psi)^{\frac{2(n-1)}{n-2}}-v_{\varepsilon}^{\frac{2(n-1)}{n-2}}\right|& \leqslant C(n)v_{\varepsilon}^{\frac{2(n-1)}{n-2}-1}|\psi|\leqslant C(n)v_{\varepsilon}^{\frac{2(n-1)}{n-2}}(\varepsilon+|x|)\\
    &\leqslant C(n)\varepsilon^{n-1}(\varepsilon+|x|)^{3-2n}.
\end{aligned}$$
\end{proof}

\subsection*{The test function for non-umbilic boundary}
Let us define 
\begin{equation}\label{phi1}
    \phi_{1} = \eta_{\delta}\cdot(v_{\varepsilon}+\psi)
\end{equation} 
as the local test function. By Proposition \ref{ve energy1} and Proposition \ref{ve energy2}, we have the following energy estimate for $\phi_1$. A similar estimate can also be obtained by a slight modification of the proof of \cite{article27}*{Corollary 3.10}.
\begin{prop}\label{phi1 int}
If $\delta_{0}$ is sufficiently small, then
$$
\begin{aligned}
	&\quad\, \int_{M}\left(\frac{4(n-1)}{n-2}\left|\nabla\phi_{1}\right|_g^2+R_g\phi_{1}^2\right) dV_{g} \\
	& \leqslant \mathcal{Q}(B^{n}, \partial B^n)\left(\int_{\partial M}\phi_{1}^{\frac{2(n-1)}{n-2}} d \sigma_{g} \right)^{\frac{n-2}{n-1}}\\
 &\quad -\theta\|\pi_{g}(p)\|^{2} \varepsilon^{n-2} \int_{B_\delta \cap \mathbb{R}_{+}^n}(\varepsilon+|x|)^{4-2 n} d x + C(\delta)\varepsilon^{n-2}
\end{aligned}
$$
for all $ 0<2\varepsilon \leqslant \delta\leqslant\delta_{0}$. Here $\theta$ is a constant depending only on $(M, g_0)$.
\end{prop}

\begin{proof}
Since $\operatorname{det}g=1+O(|x|^{2d+2}),$ we can choose small $\delta $ to make $dV_{g}$ close to $dx$ and $d\sigma_{g}$ close to $d\sigma.$ By Proposition \ref{ve energy1} and Proposition \ref{ve energy2},
$$
\begin{aligned}
       	&\quad\, \int_{\Omega_{\delta}}\left(\frac{4(n-1)}{n-2}\left|\nabla\phi_{1}\right|_g^2+R_g\phi_{1}^2\right) dV_{g} \\
	& \leqslant \mathcal{Q}(B^{n}, \partial B^n)\left(\int_{\partial M}\phi_{1}^{\frac{2(n-1)}{n-2}} d \sigma_{g} \right)^{\frac{n-2}{n-1}}+C(\delta)\varepsilon^{n-2} + J_1 +J_2+J_3+J_4+J_5,
 \end{aligned}
$$
where
$$\begin{aligned}
&J_1=C \sum_{i, k=1}^n \sum_{1 \leqslant|\alpha| \leqslant d}\left|h_{i k, \alpha}\right| \delta^{|\alpha|-n+1} \varepsilon^{n-1},\\
&J_2=C \sum_{i, k=1}^n \sum_{1 \leqslant|\alpha| \leqslant d}\left|h_{i k, \alpha}\right|^2 \varepsilon^{n-1} \delta \int_{B_\delta \cap \partial \mathbb{R}_{+}^n}(\varepsilon+|x|)^{2|\alpha|-2 n+2} d \sigma,\\
&J_3=\int_{\partial B_\delta \cap \mathbb{R}_{+}^n} \sum_{i=1}^n\left(\frac{4(n-1)}{n-2} v_\varepsilon \partial_i v_\varepsilon+\sum_{k=1}^n\left(v_\varepsilon^2 \partial_k h_{i k}-\partial_k (v_\varepsilon^2) h_{i k}\right)\right) \frac{x_i}{|x|} d \sigma,\\
&J_4=- \frac{\theta}2 \sum_{i, k=1}^n \sum_{1 \leqslant|\alpha| \leqslant d}\left|h_{i k, \alpha}\right|^2 \varepsilon^{n-2} \int_{B_\delta \cap \mathbb{R}_{+}^n}(\varepsilon+|x|)^{2|\alpha|+2-2 n} d x,\\
&J_5=C \varepsilon^{n-2} \sum_{i, k=1}^n \sum_{1 \leqslant|\alpha| \leqslant d}\left|h_{i k, \alpha}\right| \delta^{|\alpha|+2-n}.
 \end{aligned}
    $$

Since $\varepsilon<\delta$ we have
$$J_1\leqslant C\varepsilon^{n-1}\delta^{|\alpha|-n+1}<C\varepsilon^{n-2}\delta^{|\alpha|-n+2}\leqslant C(\delta)\varepsilon^{n-2}.
$$

Notice that
    $$
\begin{aligned}
    \varepsilon^{n-1} \delta \int_{B_\delta \cap \partial \mathbb{R}_{+}^n}(\varepsilon+|x|)^{2|\alpha|-2 n+2} d \sigma&=\omega_{n-2}\varepsilon^{2|\alpha|}\delta\int_{0}^{\delta/\varepsilon}(1+r)^{2|\alpha|-2n+2}r^{n-2}dr,\\
   \text{and } \quad\varepsilon^{n-2} \int_{B_\delta \cap \mathbb{R}_{+}^n}(\varepsilon+|x|)^{2|\alpha|+2-2 n} d x&=\omega_{n-1}\varepsilon^{2|\alpha|}\int_{0}^{\delta/\varepsilon}(1+r)^{2|\alpha|-2n+2}r^{n-1}dr.
\end{aligned}
    $$
Take $\delta$ small and we have
$$\begin{aligned}
&\quad\,    C\varepsilon^{n-1} \delta \int_{B_\delta \cap \partial \mathbb{R}_{+}^n}(\varepsilon+|x|)^{2|\alpha|-2 n+2} d \sigma\\
&=C\omega_{n-2}\varepsilon^{2|\alpha|}\delta\left(\int_{0}^{1}(1+r)^{2|\alpha|-2n+2}r^{n-2}dr
+\int_{1}^{\delta/\varepsilon}(1+r)^{2|\alpha|-2n+2}r^{n-2}dr\right)\\
&< \frac{\theta}{4}\omega_{n-1}\varepsilon^{2|\alpha|}\left(\int_{0}^{1}(1+r)^{2|\alpha|-2n+2}r^{n-1}dr
+\int_{1}^{\delta/\varepsilon}(1+r)^{2|\alpha|-2n+2}r^{n-1}dr\right)\\
&=\frac{\theta}{4}\varepsilon^{n-2} \int_{B_\delta \cap \mathbb{R}_{+}^n}(\varepsilon+|x|)^{2|\alpha|+2-2 n} d x.
\end{aligned}$$
This means, 
$$
J_2+J_4<\frac12 J_4.
$$
Furthermore, $\sum_{|\alpha|=1}\sum_{i,k=1}^{n}|h_{ik,\alpha}|^{2}=4\|\pi_{g}(p)\|^{2} $ by (\ref{h=pi}). So

$$
\begin{aligned}
 &\quad\,\sum_{1 \leqslant|\alpha| \leqslant d}\left(\sum_{i,k=1}^{n}\left|h_{i k, \alpha}\right|^2 \varepsilon^{n-2} \int_{B_\delta \cap \mathbb{R}_{+}^n}(\varepsilon+|x|)^{2|\alpha|+2-2 n} d x\right)\\
 &>\sum_{|\alpha|=1}\left(\sum_{i,k=1}^{n}\left|h_{i k, \alpha}\right|^2 \varepsilon^{n-2} \int_{B_\delta \cap \mathbb{R}_{+}^n}(\varepsilon+|x|)^{2|\alpha|+2-2 n} d x\right)\\
 & =4 \|\pi_{g}(p)\|^{2}\varepsilon^{n-2} \int_{B_\delta \cap \mathbb{R}_{+}^n}(\varepsilon+|x|)^{4-2 n} d x.   
\end{aligned}
$$
Thus,
$$
J_2+J_4<\frac12 J_4< -\theta\|\pi_{g}(p)\|^{2} \varepsilon^{n-2} \int_{B_\delta \cap \mathbb{R}_{+}^n}(\varepsilon+|x|)^{4-2 n} d x.
$$

On the other hand, by (\ref{ve bound}) and the Taylor expansion of $h_{ik}$, we have
$$\begin{aligned}
&J_3=\int_{\partial B_\delta \cap \mathbb{R}_{+}^n} \sum_{i=1}^n\left(\frac{4(n-1)}{n-2} v_\varepsilon \partial_i v_\varepsilon+\sum_{k=1}^n\left(v_\varepsilon^2 \partial_k h_{i k}-\partial_k (v_\varepsilon^2) h_{i k}\right)\right) \frac{x_i}{|x|} d \sigma\\
&\leqslant C\varepsilon^{n-2}\int_{\partial B_\delta \cap \mathbb{R}_{+}^n}\sum_{i,k=1}^n\sum_{1\leqslant|\alpha|\leqslant d} (\varepsilon+|x|)^{3-2 n}\left(1+(\varepsilon+|x|)|h_{ik,\alpha}|\delta^{|\alpha|-1} +|h_{ik,\alpha}|\delta^{|\alpha|}\right)  d \sigma\\
&\leqslant C\varepsilon^{n-2}\left(\delta^{3-2n}+\sum_{i,k=1}^n\sum_{1\leqslant|\alpha|\leqslant d}|h_{ik,\alpha}|\delta^{|\alpha|+3-2n}\right)\delta^{n-1}\\
&\leqslant C(\delta)\varepsilon^{n-2}.
\end{aligned}$$

In addition,
$$J_5\leqslant C\varepsilon^{n-2}\delta^{|\alpha|-n+2}\leqslant C(\delta)\varepsilon^{n-2}.
$$

Finally, we estimate the energy of $\phi_1$ outside of $\Omega_\delta$. Note $|\psi|\leqslant \frac{1}{2}v_{\varepsilon}$ as in (\ref{phi<ve}), so
$$
\left|\int_{ M\setminus\Omega_\delta}R_{g}\phi_{1}^{2}dV_g\right|=\left|\int_{\Omega_{2\delta}\setminus \Omega_{\delta}}R_{g}\phi_{1}^{2}dV_g\right|\leqslant C\int_{\Omega_{2\delta}\setminus \Omega_{\delta}} v_{\varepsilon}^{2} dV_{g}\leqslant C(\delta)\varepsilon^{n-2}.
$$
To estimate the gradient term, we differentiate (\ref{psi def}). Using the standard estimate $|\nabla^2 v_\varepsilon| \leqslant C \varepsilon^{\frac{n-2}{2}}(\varepsilon+|x|)^{-n}$ alongside the bounds for $V$ from (\ref{V property}), we obtain
$$
|\nabla \psi| \leqslant C \varepsilon^{\frac{n-2}{2}} \left( (\varepsilon+|x|)^{2-n} + (\varepsilon+|x|)^{1-n+1}+ (\varepsilon+|x|)^{2-n} \right) \leqslant C \varepsilon^{\frac{n-2}{2}} (\varepsilon+|x|)^{2-n}.
$$
Since $|x| < 2\delta \ll 1$, this implies $|\nabla \psi| \leqslant C|\nabla v_{\varepsilon}| \leqslant C \varepsilon^{\frac{n-2}{2}} (\varepsilon+|x|)^{1-n}$. Thus,
$$
\begin{aligned}
    &\quad\,\int_{M\setminus\Omega_\delta}|\nabla \phi_{1}|_{g}^{2}dV_g\\
    &\leqslant C\int_{\Omega_{2\delta}\setminus\Omega_{\delta}}\left(|\nabla \eta_{\delta}|_{g}^{2}(v_{\varepsilon}+\psi)^{2}+|\nabla(v_{\varepsilon}+\psi)|_{g}^{2}\eta_{\delta}^{2}\right)dV_{g}\\
    &\leqslant C\varepsilon^{n-2}\int_{B_{2\delta}\setminus B_{\delta}}\left(\delta^{-2}(\varepsilon+|x|)^{4-2n}+(\varepsilon+|x|)^{2-2n}\right)dx\\
    &\leqslant C(\delta)\varepsilon^{n-2}.
\end{aligned}
$$
This means,
$$
\int_{M\setminus\Omega_{\delta}}\left(\frac{4(n-1)}{n-2}\left|\nabla\phi_{1}\right|_g^2+R_g\phi_{1}^2\right) dV_{g}\leqslant C(\delta)\varepsilon^{n-2}.
$$

Putting them together gives the proposition.
\end{proof}

\subsection*{The test function for umbilic boundary}
For manifolds with umbilic boundary, since $h_{g}=0$, we have $\pi_{g}=0$. So by (\ref{h=pi}),
$$\sum_{|\alpha|=1}\sum_{i,k=1}^{n}|h_{ik,\alpha}|^{2}={4}\|\pi_{g}(0)\|^{2}=0.$$ 
Thus in the Taylor expansion of $h$, the subscript starts from $|\alpha|=2$. 

We make use of $G: M\setminus\{p\}\rightarrow\mathbb{R}$ the Green's
function for the conformal Laplacian with Neumann boundary condition with pole at $p$. More precisely, $G$ satisfies
\begin{equation*}
	\begin{cases}\frac{4(n-1)}{n-2} \Delta_{g} G-R_{g} G=0\qquad& \text{ in } M \setminus\left\{p\right\}, \\  G_\nu=0& \text{ on } \partial M \setminus\left\{p\right\}.\end{cases}
\end{equation*}
We also assume $G$ is normalized such that $\lim_{|x|\to 0}|x|^{n-2}G(x)=1$. Then $G$ satisfies
\begin{equation}\label{G}
    \left|G(x)-| x|^{2-n}\right|\leqslant C \sum_{2 \leqslant|\alpha| \leqslant d} \sum_{i, k=1}^{n}|h_{i k, \alpha}|\cdot| x|^{|\alpha|+2-n}+C|x|^{d+3-n}.
\end{equation}
Moreover, we consider as in \cite{article4} the flux integral
$$
\begin{aligned}
	\mathcal{I}(p, \delta)= & \frac{4(n-1)}{n-2} \int_{\partial B_\delta \cap \mathbb{R}_{+}^n} \sum_{i=1}^n\left(|x|^{2-n} \partial_i G-G \partial_i|x|^{2-n}\right) \frac{x_i}{|x|} dx \\
	& -\int_{\partial B_\delta \cap \mathbb{R}_{+}^n} \sum_{i, k=1}^n|x|^{2-2 n}\left(|x|^2 \partial_k h_{i k}-2 n x_k h_{i k}\right) \frac{x_i}{|x|} dx,
\end{aligned}
$$
where $\delta>0$ is sufficiently small.

Finally, we define
\begin{equation}\label{phi2}
\phi_{2}=\eta_{\delta}\cdot(v_{\varepsilon}+\psi)+(1-\eta_{\delta})\varepsilon^{\frac{n-2}{2}}G
\end{equation}
as the global test function. By \cite{article4}*{Proposition 4.1}, we have the following energy estimate for $\phi_2$:
\begin{prop}\label{phi2 int}
If $\delta_{0}$ is sufficiently small, then
    	$$
	\begin{aligned}
		& \quad\,\int_{M}\left(\frac{4(n-1)}{n-2}\left|\nabla \phi_{2}\right|_g^2+R_g \phi_{2}^2\right) d V_g \\
		& \leqslant\mathcal{Q}(B^{n}, \partial B^{n})\left(\int_{\partial M} \phi_{2}^{\frac{2(n-1)}{n-2}} d \sigma_g\right)^{\frac{n-2}{n-1}} -\varepsilon^{n-2} \mathcal{I}(p, \delta)\\
  &\quad-\frac{\theta}{2} \sum_{i, k=1}^n \sum_{2 \leqslant|\alpha| \leqslant d}\left|h_{i k, \alpha}\right|^2 \varepsilon^{n-2} \int_{B_\delta \cap \mathbb{R}_{+}^n}(\varepsilon+|x|)^{2|\alpha|+2-2 n} d x \\
		& \quad+C \varepsilon^{n-2} \sum_{i, k=1}^n \sum_{2 \leqslant|\alpha| \leqslant d}\left|h_{i k, \alpha}\right| \delta^{-n+2+|\alpha|} +C \varepsilon^{n-2} \delta^{2 d+4-n}+C \delta^{-n} \varepsilon^{n}
	\end{aligned}
	$$
for all $ 0<2\varepsilon \leqslant \delta\leqslant\delta_{0}$. Here $\theta$ is a constant depending only on $(M, g_0)$.
\end{prop}

\subsection*{Some estimates}
For the convenience of the subsequent sections, we provide proofs of several lemmas related to the estimates of test functions.
\begin{lem}\label{phi1 est}
If $\delta_{0}$ is sufficiently small, then
    $$\phi_i^{\frac{2(n-1)}{n-2}}\leqslant C(n)\varepsilon^{n-1}(\varepsilon+|x|)^{2-2n}$$
for $i=1,2$, and for all $ 0<2\varepsilon \leqslant \delta\leqslant\delta_{0}$ and $|x|<\delta$.
\end{lem}

\begin{proof}
By Lemma \ref{ve est}, $|\psi|\leqslant \frac{1}{2}v_{\varepsilon}$. So when $|x|<\delta$, $\phi_1=\phi_2 = v_{\varepsilon}+\psi\geqslant \frac{1}{2}v_{\varepsilon}>0$. Thus by Lemma \ref{ve est} and (\ref{ve bound}),
$$\begin{aligned}
\phi_i^{\frac{2(n-1)}{n-2}}& \leqslant \left|(v_{\varepsilon}+\psi)^{\frac{2(n-1)}{n-2}}-v_{\varepsilon}^{\frac{2(n-1)}{n-2}}\right| + \left|v_{\varepsilon}^{\frac{2(n-1)}{n-2}}\right|\\
& \leqslant C\varepsilon^{n-1}(\varepsilon+|x|)^{3-2n} + C\varepsilon^{n-1}(\varepsilon+|x|)^{2-2n}\\
& \leqslant C\varepsilon^{n-1}(\varepsilon+|x|)^{2-2n}.
	\end{aligned}$$
\end{proof}

\begin{lem}\label{phi est outside}
If $\delta_{0}$ is sufficiently small, then
    $$
	\int_{\partial M\setminus\Omega_\delta}\phi_{i}^{\frac{2(n-1)}{n-2}}d\sigma_g=O(\varepsilon^{n-1}\delta^{2-2n})
	$$
for $i=1,2$, and for all $ 0<2\varepsilon \leqslant \delta\leqslant\delta_{0}$.
\end{lem}

\begin{proof}
We only need to prove this result for $\phi_{2}$, as $0\leqslant\phi_1\leqslant\phi_{2}$.

    On $\Omega_{2\delta}\setminus\Omega_{\delta}$,
$$
\phi_2=\eta_{\delta}\left(v_{\varepsilon}+\psi\right)+(1-\eta_{\delta})\varepsilon^{\frac{n-2}{2}}G.
$$
From (\ref{G}) we have
 	$$
  \begin{aligned}
      G & \leqslant  |x|^{2-n}+C|x|^{4-n}+C|x|^{5-n}+\cdots+C|x|^{3-n+d}\\
      &\leqslant  \delta^{2-n}+C\left(\delta^{4-n}+\delta^{5-n}+\cdots+\delta^{3-n+d}\right).
  \end{aligned}$$
So we can take $\delta$ small enough and get $G\leqslant  C\delta^{2-n}.$ Thus, together with $|\psi|\leqslant \frac{1}{2}v_{\varepsilon}$ we have
$$\begin{aligned}
\phi_{2}&\leqslant  Cv_{\varepsilon}+C\varepsilon^{\frac{n-2}{2}}G\\
		&=  C\varepsilon^{\frac{n-2}{2}}\left(\left((\varepsilon+x_n)^{2}+\sum_{a=1}^{n-1}x_a^{2}\right)^{-\frac{n-2}{2}}+CG\right)\\
		&\leqslant  C\varepsilon^{\frac{n-2}{2}}\left(\left(\varepsilon^{2}+|x|^{2}\right)^{-\frac{n-2}{2}}+C\delta^{2-n}\right)\\
		&\leqslant  C\varepsilon^{\frac{n-2}{2}}\delta^{2-n}.
\end{aligned}$$

  On $\partial M\setminus\Omega_{2\delta}$, 
  $$\phi_{2}=\varepsilon^{\frac{n-2}{2}}G.$$
  We can still take $\delta$ small enough and get $G\leqslant  C\delta^{2-n}$, thus
	$$\phi_{2}\leqslant  C\varepsilon^{\frac{n-2}{2}}\delta^{2-n}.$$

Therefore,
 	$$
	\int_{\partial M\setminus\Omega_\delta}\phi_{2}^{\frac{2(n-1)}{n-2}}\leqslant  C\varepsilon^{n-1}\delta^{-2(n-1)}
	$$
for $ 0<2\varepsilon \leqslant \delta\leqslant\delta_{0}$.

\end{proof}

\begin{lem}\label{phi est inside}
If $\delta_{0}$ is sufficiently small, then
    $$\int_{\partial M}\phi_{i}^{\frac{2(n-1)}{n-2}}d\sigma_{g}=A(n)+O(\varepsilon)+O(\varepsilon^{n-1}\delta^{1-n})+O(\varepsilon^{n-1}\delta^{2-2n})$$
for $i=1,2$, and for all $ 0<2\varepsilon \leqslant \delta\leqslant\delta_{0}$. Here
\begin{equation}\label{A}
    A(n)=\omega_{n-2}\int_{0}^{\infty}\frac{r^{n-2}}{\left(1+r^{2}\right)^{n-1}}dr
\end{equation}
is a constant which only depends on $n$.
\end{lem}

\begin{proof}
By Lemma \ref{phi est outside}, it suffices to prove the following local estimate
$$\int_{B_{\delta}\cap\partial\mathbb{R}_{+}^{n}}(v_{\varepsilon}+\psi)^{\frac{2(n-1)}{n-2}}d\sigma=A(n)+O(\varepsilon)+O(\varepsilon^{n-1}\delta^{1-n}).$$

    By Lemma \ref{ve est},
    	$$\begin{aligned}
&\quad\,\left|\int_{B_{\delta}\cap\partial\mathbb{R}_{+}^{n}}(v_{\varepsilon}+\psi)^{\frac{2(n-1)}{n-2}}d\sigma-\int_{B_{\delta}\cap\partial\mathbb{R}_{+}^{n}}v_{\varepsilon}^{\frac{2(n-1)}{n-2}}d\sigma\right|\\
		&\leqslant C\int_{B_{\delta}\cap\partial\mathbb{R}_{+}^{n}} \varepsilon^{n-1}(\varepsilon+|x|)^{3-2n}d\sigma\\
		&\leqslant C\varepsilon \int_{B_{\delta/\varepsilon}^{n-1}}(1+|x|)^{3-2n}dx\\
		&\leqslant C\varepsilon\int_{0}^{\delta/\varepsilon}\frac{r^{n-2}}{(1+r)^{2n-3}}dr\leqslant C\varepsilon.
	\end{aligned}$$
While
 	$$\begin{aligned}
		\int_{B_{\delta}\cap\partial \mathbb{R}_{+}^n}v_{\varepsilon}^{\frac{2(n-1)}{n-2}}d\sigma&=\int_{B_{\delta}^{n-1}}\varepsilon^{n-1}\left(\varepsilon^{2}+|x|^{2}\right)^{-(n-1)}dx\qquad (\text{since } x_n=0)\\
		&=\omega_{n-2}\int_{0}^{\delta/\varepsilon}\frac{r^{n-2}}{\left(1+r^{2}\right)^{n-1}}dr\\
		&=\omega_{n-2}\int_{0}^{\infty}\frac{r^{n-2}}{\left(1+r^{2}\right)^{n-1}}dr-\omega_{n-2}\int_{\delta/\varepsilon}^{\infty}\frac{r^{n-2}}{\left(1+r^{2}\right)^{n-1}}dr,
	\end{aligned}$$
where
$$\int_{\delta/\varepsilon}^{\infty}\frac{r^{n-2}}{\left(1+r^{2}\right)^{n-1}}dr\leqslant \int_{\delta/\varepsilon}^{\infty}r^{-n}dr\leqslant C(\delta/\varepsilon)^{1-n}.$$
  So
$$\int_{B_{\delta}\cap\partial\mathbb{R}_{+}^{n}}(v_{\varepsilon}+\psi)^{\frac{2(n-1)}{n-2}}d\sigma=A(n)+O(\varepsilon)+O(\varepsilon^{n-1}\delta^{1-n}).$$
\end{proof}

\begin{lem}\label{f est}
    Let $A=A(n)$ as given in (\ref{A}). Let $\alpha\in(0,1)$ be a constant, $f_{1}, f_{2}$ be functions defined in a neighborhood of $0$, with $f_{i}>A>0$ for $i=1,2$ and 
    $ \lim_{x\to 0}\left(f_{1}-f_{2}\right)(x)=0.$ Then there exists a constant $C(\alpha, n)$ such that
    $$f_{1}^{\alpha}\leqslant f_{2}^{\alpha}+C(\alpha, n)|f_{1}-f_{2}|$$
    for $|x|\ll1$.
\end{lem}

\begin{proof}
By the Mean Value Theorem,
    	$$\left|f_{1}^{\alpha}(x)-f_{2}^{\alpha}(x)\right|=\alpha\left|f_{1}(x)-f_{2}(x)\right|\Big|f_{2}(x)+\gamma(x)\left(f_{1}(x)-f_{2}(x)\right)\Big|^{\alpha-1}$$
for some $\gamma(x)\in(0,1)$. When $|x|\ll1$, we have $|f_{1}(x)-f_{2}(x)|<\frac{A}{2}$, and $$f_{2}(x)+\gamma(x)\left(f_{1}(x)-f_{2}(x)\right)>\frac{A}{2}.$$
Thus,
   $$f_{1}(x)^{\alpha}\leqslant f_{2}(x)^{\alpha}+C(\alpha,A)|f_{1}(x)-f_{2}(x)|.$$
\end{proof}

\section{Manifolds with non-umbilic boundary}\label{sec4}
In this section we focus on manifolds with non-umbilic boundary. Moreover, we assume the function $f$ is somewhere positive and achieves a global maximum at a non-umbilic point $p\in\partial M$. By rescaling the function we may assume that $\displaystyle 1=f(p)=\max_{\partial M} f$. 

When $n\geqslant 5$, we assume $f$ satisfies condition (\ref{**}), that is, we further assume the function $f$ satisfies $\Delta_{\partial M} f(p)\leqslant c(n)\|\pi_{g}-h_{g} g\|^2(p)$. As Escobar pointed out in \cite{article15}, condition (\ref{**}) is conformally invariant. Thus if we take a conformal metric with $h_{g} =0$, the inequality becomes 
$$\Delta_{\partial M} f(p)\leqslant c(n)\|\pi_{g}\|^{2}(p).$$

\begin{proof}[Proof of Theorem \ref{n=5, nonumbilic}]
We will use the local test function $\phi_{1}$ defined as (\ref{phi1}). Recall from Proposition \ref{phi1 int}, for $\delta_{0}$ sufficiently small and for all $ 0<2\varepsilon \leqslant \delta\leqslant\delta_{0}$,
$$
\begin{aligned}
	E_{g}(\phi_{1}) & \leqslant  \mathcal{Q}(B^{n}, \partial B^n)\left(\int_{\partial M}\phi_{1}^{\frac{2(n-1)}{n-2}} d \sigma_{g} \right)^{\frac{n-2}{n-1}}\\
 &\quad -\theta\|\pi_{g}(p)\|^{2} \varepsilon^{n-2} \int_{B_\delta \cap \mathbb{R}_{+}^n}(\varepsilon+|x|)^{4-2 n} d x + C(\delta)\varepsilon^{n-2}.
\end{aligned}
$$
Our main goal is to replace $\left(\int_{\partial M}\phi_{1}^{\frac{2(n-1)}{n-2}} d \sigma_{g} \right)^{\frac{n-2}{n-1}}$ by $\left(\int_{\partial M}f\phi_{1}^{\frac{2(n-1)}{n-2}} d \sigma_{g} \right)^{\frac{n-2}{n-1}}$ and estimate their difference.

When $n\geqslant 5$, we have 
    $$\begin{aligned}
    \varepsilon^{n-2} \int_{B_\delta \cap \mathbb{R}_{+}^n}(\varepsilon+|x|)^{4-2 n} dx = \varepsilon^{2}\omega_{n-1}\int_{0}^{\delta/\varepsilon}\frac{r^{n-1}}{(1+r)^{2n-4}}dr>B(n)\varepsilon^{2},
    \end{aligned}$$
where 
$$B(n)=\omega_{n-1}\int_{0}^{1}\frac{r^{n-1}}{(1+r)^{2n-4}}dr$$ 
is a constant which only depends on $n$. So
\begin{equation}\label{E phi1}
    E_{g}(\phi_{1})\leqslant \mathcal{Q}(B^{n}, \partial B^n)\left(\int_{\partial M}\phi_{1}^{\frac{2(n-1)}{n-2}} d \sigma_{g} \right)^{\frac{n-2}{n-1}}-\theta B(n)\|\pi_{g}(p)\|^{2}\varepsilon^{2} + C(\delta)\varepsilon^{n-2}.
\end{equation}

For $|x|\ll1$, we have $f(x)-1=\frac{1}{2}\sum_{i,j=1}^{n-1}f_{ij}(0)x_{i}x_{j}+O(|x|^{3})$. Then
$$\begin{aligned}
&\quad\,\int_{\partial M}\phi_{1}^{\frac{2(n-1)}{n-2}} d \sigma_{g}\\
&=\int_{\partial M}f\phi_{1}^{\frac{2(n-1)}{n-2}} d \sigma_{g}+\int_{\partial M}(1-f)\phi_{1}^{\frac{2(n-1)}{n-2}} d \sigma_{g}\\
&=\int_{\partial M}f\phi_{1}^{\frac{2(n-1)}{n-2}} d \sigma_{g}+\int_{\Omega_{\delta}\cap\partial M}(1-f)\phi_{1}^{\frac{2(n-1)}{n-2}} d \sigma_{g}+E_{1}\\
&=\int_{\partial M}f\phi_{1}^{\frac{2(n-1)}{n-2}} d \sigma_{g}-\frac{1}{2}\sum_{i,j=1}^{n-1}\int_{\Omega_{\delta}\cap\partial M}f_{ij}(0)x_{i}x_{j}\phi_{1}^{\frac{2(n-1)}{n-2}}d\sigma_{g}+E_{1}+E_{2}\\
&=\int_{\partial M}f\phi_{1}^{\frac{2(n-1)}{n-2}} d \sigma_{g}-\frac{1}{2}\sum_{i,j=1}^{n-1}\int_{B_{\delta}\cap\partial\mathbb{R}_{+}^{n}}f_{ij}(0)x_{i}x_{j}v_{\varepsilon}^{\frac{2(n-1)}{n-2}}d\sigma+E_{1}+E_{2}+E_{3}+O(\varepsilon^{3}),
\end{aligned}$$
where $O(\varepsilon^{3})$ comes from the conformal change of the coordinates, and
$$
\begin{aligned}
E_{1} & =\int_{\partial M\setminus\Omega_{\delta}}(1-f)\phi_{1}^{\frac{2(n-1)}{n-2}}d\sigma_{g},\\
E_{2} & =\int_{\Omega_{\delta}\cap\partial M}O(|x|^{3}) \phi_{1}^{\frac{2(n-1)}{n-2}}d\sigma_{g},\\
E_{3} & = -\frac12\sum_{i,j=1}^{n-1}\int_{B_{\delta}\cap\partial\mathbb{R}_{+}^{n}}\left[(v_{\varepsilon}+\psi)^{\frac{2(n-1)}{n-2}}-v_{\varepsilon}^{\frac{2(n-1)}{n-2}}\right]f_{ij}(0)x_{i}x_{j}d\sigma.\\
\end{aligned}
$$

By the symmetry of the ball, we get 
$$
\begin{aligned}
    &\quad\,-\frac{1}{2}\sum_{i,j=1}^{n-1}\int_{B_{\delta}\cap\partial\mathbb{R}_{+}^{n}}f_{ij}(0)x_{i}x_{j}v_{\varepsilon}^{\frac{2(n-1)}{n-2}}d\sigma\\
    & =-\frac{1}{2}\sum_{i=1}^{n-1}f_{ii}(0)\int_{B_{\delta}\cap\partial\mathbb{R}_{+}^{n}}x_{i}^2\,\varepsilon^{n-1}(\varepsilon^{2}+|x|^{2})^{1-n}d\sigma \quad (\text{since } x_n=0)\\
    & =\frac12\Delta_{\partial M} f(0)\int_{B_{\delta}\cap\partial\mathbb{R}_{+}^{n}}\frac{|x|^{2}}{n-1}\varepsilon^{n-1}(\varepsilon^{2}+|x|^{2})^{1-n}d\sigma\\
    &=\frac{\Delta_{\partial M} f(0)\varepsilon^{2}}{2(n-1)}\omega_{n-2}\int_{0}^{\delta/\varepsilon}\frac{r^{n}}{(1+r^{2})^{n-1}}dr<D(n)\Delta_{\partial M} f(0)\varepsilon^{2},
\end{aligned}
$$
where $D(n)=\frac{\omega_{n-2}}{2(n-1)}\int_{0}^{\infty}\frac{r^{n}}{(1+r^{2})^{n-1}}dr<\infty$ as long as $n>3$.\\
By Lemma \ref{phi est outside},
$$
|E_{1}|\leqslant C\varepsilon^{n-1}\delta^{2-2n}\leqslant C\varepsilon^{3}.
$$
By Lemma \ref{phi1 est},
$$
\begin{aligned}
    |E_{2}|\leqslant C\int_{B_{\delta}\cap\partial\mathbb{R}_{+}^{n}}\varepsilon^{n-1}(\varepsilon+|x|)^{2-2n}|x|^{3}d\sigma\leqslant C\varepsilon^{3}\int_{0}^{\delta/\varepsilon}\frac{r^{n+1}}{(1+r)^{2n-2}}dr\leqslant C\varepsilon^{3}.
\end{aligned}
$$
By Lemma \ref{ve est},
$$
\begin{aligned}
    |E_{3}|& \leqslant C\sum_{i,j=1}^{n-1}\int_{B_{\delta}\cap\partial\mathbb{R}_{+}^{n}}\varepsilon^{n-1}(\varepsilon+|x|)^{3-2n}|f_{ij}(0)|\cdot|x_{i}x_{j}|d\sigma\\ 
    & \leqslant C\int_{B_{\delta}\cap\partial\mathbb{R}_{+}^{n}}\varepsilon^{n-1}(\varepsilon+|x|)^{3-2n}|x|^{2}d\sigma\\
    & \leqslant C\varepsilon^{3}\int_{0}^{\delta/\varepsilon}\frac{r^{n}}{(1+r)^{2n-3}}dr \quad\leqslant C\varepsilon^{3}.
\end{aligned}
$$
Note the above two integrals in $E_2, E_3$ are finite as long as $n>4$. Thus, when we combine everything above, we obtain
$$
\left|\int_{\partial M}\phi_{1}^{\frac{2(n-1)}{n-2}} d \sigma_{g} - \int_{\partial M}f\phi_{1}^{\frac{2(n-1)}{n-2}} d \sigma_{g}\right|\leqslant D(n)\Delta_{\partial M} f(p)\varepsilon^{2} + C\varepsilon^{3}.
$$

Denote $\alpha = \frac{n-2}{n-1}$ and 
$$
f_1(\varepsilon) = \int_{\partial M}\phi_{1}^{\frac{2(n-1)}{n-2}} d \sigma_{g},\quad f_2(\varepsilon) = \int_{\partial M}f\phi_{1}^{\frac{2(n-1)}{n-2}} d \sigma_{g}.
$$
By Lemma \ref{f est} we get
$$
\begin{aligned}
    \left(\int_{\partial M}\phi_{1}^{\frac{2(n-1)}{n-2}} d \sigma_{g} \right)^{\frac{n-2}{n-1}}\leqslant \left(\int_{\partial M}f\phi_{1}^{\frac{2(n-1)}{n-2}} d \sigma_{g} \right)^{\frac{n-2}{n-1}}+C(n)D(n)\Delta_{\partial M} f(p)\varepsilon^{2}+C\varepsilon^{3}.
\end{aligned}
$$

Therefore, from (\ref{E phi1}) we get
$$
\begin{aligned}
	E_{g}(\phi_{1}) 
	\leqslant  & \mathcal{Q}(B^{n}, \partial B^n)\left(\int_{\partial M}f\phi_{1}^{\frac{2(n-1)}{n-2}} d \sigma_{g} \right)^{\frac{n-2}{n-1}}\\
 & +\left(\mathcal{Q}(B^{n}, \partial B^n)C(n)D(n)\Delta_{\partial M} f(p)-\theta B(n)\|\pi_{g}(p)\|^{2}\right)\varepsilon^{2} + C\varepsilon^{3}.
\end{aligned}$$
Take $\bar{c}(n) = \frac{\theta B(n)}{\mathcal{Q}(B^{n}, \partial B^n)C(n)D(n)}$, and note $\theta$ is a constant depending only on $(M, g_0)$. When $\Delta_{\partial M} f(p)< \bar{c}(n)\|\pi_{g}\|^{2}(p)$, we have
$$
\mathcal{Q}(B^{n}, \partial B^n)C(n)D(n)\Delta_{\partial M} f(p)-\theta B(n)\|\pi_{g}(p)\|^{2}<0,
$$
so we can choose $\varepsilon\ll1$ to make
\begin{equation}\label{Ephi1}
    E_{g}(\phi_{1})<\mathcal{Q}(B^{n},\partial B^{n})\left(\int_{\partial M}f\phi_{1}^{\frac{2(n-1)}{n-2}} d \sigma_{g} \right)^{\frac{n-2}{n-1}}.
\end{equation}
The function 
$$u=\left(\int_{\partial M} f \phi_{1}^{\frac{2(n-1)}{n-2}} d \sigma_g\right)^{-\frac{n-2}{2(n-1)}}\cdot \phi_{1}$$ 
then lies in the constraint set $C_{\frac{n}{n-2},f}$ and by (\ref{Ephi1}),
$$
E_{g}(u) = \left(\int_{\partial M} f \phi_{1}^{\frac{2(n-1)}{n-2}} d \sigma_g\right)^{-\frac{n-2}{n-1}}\cdot E_{g}(\phi_{1}) < \mathcal{Q}(B^{n},\partial B^{n}).
$$
This means
 $$
 E\left(\frac{n}{n-2}, f\right)<\mathcal{Q}(B^{n},\partial B^{n}).
 $$
Together with Theorem \ref{criterion}, this completes the proof of Theorem \ref{n=5, nonumbilic}.
\end{proof}
\begin{rmk}
    As long as (\ref{Ephi1}) holds for the test function $\phi_1$ when $\varepsilon \ll 1$, we can apply the above proof and use Theorem \ref{criterion} to establish the existence of a conformal metric. Therefore, our primary focus moving forward will be to show (\ref{Ephi1}).
\end{rmk}

\begin{proof}[Proof of Theorem \ref{n=4, nonumbilic}]
When $n=4$,
 $$\begin{aligned}
    & \quad \varepsilon^{n-2} \int_{B_\delta \cap \mathbb{R}_{+}^n}(\varepsilon+|x|)^{4-2 n} dx\\
    & = \varepsilon^{2}\omega_{3}\int_{0}^{\delta/\varepsilon}\frac{r^{3}}{(1+r)^{4}}dr\\
    & > \varepsilon^{2}\omega_{3}\int_{1}^{\delta/\varepsilon}\frac{r^{3}}{(1+r)^{4}}dr\\
    & \geqslant \varepsilon^{2}\omega_{3}\int_{1}^{\delta/\varepsilon}\frac{1}{16r}dr =
    C\varepsilon^{2}\log(\delta/\varepsilon).
    \end{aligned}$$
So by Proposition \ref{phi1 int}, for $\delta_{0}$ sufficiently small and for all $ 0<2\varepsilon \leqslant \delta\leqslant\delta_{0}$,
\begin{equation}\label{E; n=4}
 E_{g}(\phi_{1})\leqslant \mathcal{Q}(B^{4}, \partial B^4)\left(\int_{\partial M}\phi_{1}^{3} d \sigma_{g} \right)^{\frac{2}{3}}-C\theta \|\pi_{g}(p)\|^{2}\varepsilon^{2}\log(\delta/\varepsilon) + C(\delta)\varepsilon^{2}.
\end{equation}

For $|x|\ll1$, we have $f(x)=1+O(|x|^{2})$, so
$$
\int_{\partial M}\phi_{1}^{3} d \sigma_{g}=\int_{\partial M}f\phi_{1}^{3} d \sigma_{g}+E_{4}+E_{5},
$$
where
$$
\begin{aligned}
E_{4} & = \int_{\Omega_{\delta}}O(|x|^{2})\phi_{1}^{3} d \sigma_{g},\\
E_{5} & = \int_{M\setminus\Omega_{\delta}}(1-f)\phi_{1}^{3} d \sigma_{g}.
\end{aligned}$$
By Lemma \ref{phi1 est},
$$
\begin{aligned}
    |E_{4}|\leqslant C\int_{B_{\delta}\cap\partial\mathbb{R}_{+}^{4}}\varepsilon^{3}(\varepsilon+|x|)^{-6}|x|^{2}d\sigma\leqslant C\varepsilon^{2}\int_{0}^{\delta/\varepsilon}\frac{r^{4}}{(1+r)^{6}}dr\leqslant C\varepsilon^{2}.
\end{aligned}
$$
By Lemma \ref{phi est outside},
$$|E_{5}|\leqslant C\varepsilon^{3}\delta^{-6}\leqslant C\varepsilon^{2}.$$
Thus, when we combine everything above, we obtain
$$
\left|\int_{\partial M}\phi_{1}^{3} d \sigma_{g}-\int_{\partial M}f\phi_{1}^{3} d \sigma_{g}\right|\leqslant C\varepsilon^{2}.
$$

Therefore, by (\ref{E; n=4}) and Lemma \ref{f est} we get
$$
\begin{aligned}
 E_{g}(\phi_{1}) & \leqslant \mathcal{Q}(B^{4}, \partial B^4)\left(\int_{\partial M}f\phi_{1}^{3} d \sigma_{g} \right)^{\frac{2}{3}}\\
 & \quad -\left(C\theta\|\pi_{g}(p)\|^{2}\log(\delta/\varepsilon) - C(\delta) - C\mathcal{Q}(B^{4}, \partial B^4)\right)\varepsilon^{2}.
\end{aligned}
$$
For fixed $\delta$, we can take $\varepsilon\ll1$ to make
$$
C\theta\|\pi_{g}(p)\|^{2}\log(\delta/\varepsilon) - C(\delta) - C\mathcal{Q}(B^{4}, \partial B^4)>0,
$$
and as a result,
$$E_{g}(\phi_{1})<\mathcal{Q}(B^{4}, \partial B^4)\left(\int_{\partial M}f\phi_{1}^{3} d \sigma_{g} \right)^{\frac{2}{3}}.$$
\end{proof}

\section{Manifolds with umbilic boundary}\label{sec5}
In this section we focus on manifolds with umbilic boundary for $n\geqslant6$. Moreover, we assume the function $f$ is somewhere positive and achieves a global maximum at $p\in\partial M$, where $\nabla^k f(p)=0$ for $k=1, \cdots, n-2$. By rescaling the function we may assume that $\displaystyle 1=f(p)=\max_{\partial M} f$. 

For manifolds with umbilic boundary, we have $\pi_{g_0} \equiv 0$ since $h_{g_0}=0$. Under the conformal change $g = u^{\frac{4}{n-2}}g_0$, the boundary remains totally umbilic, meaning the second fundamental form satisfies $\pi_{ab}(x) = h_g(x) g_{ab}(x)$ for all $x \in \partial M$. As established in Section~\ref{sec2}, the conformal Fermi coordinates ensure $h_g(x) = O(|x|^{2d+1})$. Therefore, $\pi_{ab}(x) = O(|x|^{2d+1})$. 

In these coordinates, $\partial_n g_{ab} = -2\pi_{ab}$, which implies $\partial_n h_{ab}(x) = O(|x|^{2d+1})$ on $\partial \mathbb{R}_{+}^n$. Because $H_{ik}(x)$ is defined as a polynomial of degree $d$, its normal derivative $\partial_n H_{ik}(x)$ is a polynomial of degree at most $d-1$. Since the function it approximates vanishes to order $2d+1$, the polynomial must be identically zero, yielding $\partial_n H_{ik}(x) = 0$ on $\partial \mathbb{R}_{+}^n$. 

In particular, evaluated at the origin, we have $\pi_g(0)=0$. So by \eqref{h=pi},
$$\sum_{|\alpha|=1}\sum_{i,k=1}^{n}|h_{ik,\alpha}|^{2}={4}\|\pi_{g}(0)\|^{2}=0.$$ 
Thus in the Taylor expansion of $h$, the subscript starts from $|\alpha|=2$.

We will make use of the local test function $\phi_{2}$ defined as (\ref{phi2}). Let us start with the following estimate.

\begin{prop}\label{5.1}
If $\delta_{0}$ is sufficiently small, then
    $$
	\begin{aligned}
		&\quad\, \int_M\left(\frac{4(n-1)}{n-2}\left|\nabla \phi_{2}\right|_g^2+R_g \phi_{2}^2\right) d V_g \\
		&\leqslant \mathcal{Q}(B^{n}, \partial B^{n})\left(\int_{\partial M} f\phi_{2}^{\frac{2(n-1)}{n-2}} d \sigma_g\right)^{\frac{n-2}{n-1}}\\
  &\quad-\frac{\theta}{2} \sum_{i, k=1}^n \sum_{2 \leqslant|\alpha| \leqslant d}\left|h_{i k, \alpha}\right|^2 \varepsilon^{n-2} \int_{B_\delta \cap \mathbb{R}_{+}^n}(\varepsilon+|x|)^{2|\alpha|+2-2 n} d x \\
		&\quad - \varepsilon^{n-2} \mathcal{I}(p, \delta)+C \varepsilon^{n-2} \sum_{i, k=1}^n \sum_{2 \leqslant|\alpha| \leqslant d}\left|h_{i k, \alpha}\right| \delta^{-n+2+|\alpha|}+C \varepsilon^{n-2} \delta^{2 d+4-n}\\
		&\quad+C\varepsilon^{n-1}\log(\delta/\varepsilon)+C\varepsilon^{n-1}\delta^{2-2n}+C\varepsilon^{n}\delta^{-n}.
	\end{aligned}
	$$
for all $ 0<2\varepsilon \leqslant \delta\leqslant\delta_{0}$. Here $\theta$ is a constant depending only on $(M, g_0)$.
\end{prop}

\begin{proof}
Our main goal is to replace $\left(\int_{\partial M}\phi_{2}^{\frac{2(n-1)}{n-2}} d \sigma_{g} \right)^{\frac{n-2}{n-1}}$ by $\left(\int_{\partial M}f\phi_{2}^{\frac{2(n-1)}{n-2}} d \sigma_{g} \right)^{\frac{n-2}{n-1}}$ in Proposition \ref{phi2 int} and estimate their difference. We have
$$
\begin{aligned}
    	\int_{\partial M}\phi_{2}^{\frac{2(n-1)}{n-2}}d\sigma_g & =\int_{\partial M}\left[f\phi_{2}^{\frac{2(n-1)}{n-2}}+(1-f)\phi_{2}^{\frac{2(n-1)}{n-2}}\right]d\sigma_{g}\\
 &=\int_{\partial M}f\phi_{2}^{\frac{2(n-1)}{n-2}}d\sigma_{g}+E_{1}+E_{2},
\end{aligned}$$
where
$$\begin{aligned}
E_{1} & = \int_{\Omega_\delta \cap \partial M}(1-f)\phi_{2}^{\frac{2(n-1)}{n-2}}d\sigma_g,\\
E_{2} & = \int_{\partial M\setminus\Omega_\delta}(1-f)\phi_{2}^{\frac{2(n-1)}{n-2}}d\sigma_{g}.
\end{aligned}$$

For $\left|x\right|<\delta\ll1$, we have $\left|1-f(x)\right| \leqslant C\left|\nabla^{n-1}f(x_0)\right|\cdot\left|x\right|^{n-1}$ for some $x_0$ between $x$ and $0$. So by Lemma \ref{phi1 est},
$$\begin{aligned}
\left|E_{1}\right|&=\left|\int_{\Omega_\delta \cap \partial M}(1-f)\phi_{1}^{\frac{2(n-1)}{n-2}}d\sigma_g\right|\\
& \leqslant  C\int_{\Omega_{\delta}\cap \partial M} |x|^{n-1}\varepsilon^{n-1}(\varepsilon+|x|)^{2-2n}d\sigma_{g}\\
		&\leqslant  C\int_{B_{\delta}\cap\partial\mathbb{R}_{+}^{n}} |x|^{n-1}\varepsilon^{n-1}\left(\varepsilon^{2}+|x|^{2}\right)^{-(n-1)}d\sigma + C\varepsilon^{n-1}\quad (\text{since } x_n=0)\\
  &\leqslant  C\varepsilon^{n-1}\int_{0}^{\delta/\varepsilon}\frac{r^{2n-3}}{\left(1+r^{2}\right)^{n-1}}dr =  C\varepsilon^{n-1}\left(\int_{0}^{1}+\int_{1}^{\delta/\varepsilon}\right)\frac{r^{2n-3}}{\left(1+r^{2}\right)^{n-1}}dr\\
  &\leqslant  C\varepsilon^{n-1}\left(A+\log(\delta/\varepsilon)\right)\\
  & \leqslant  C\varepsilon^{n-1}\log(\delta/\varepsilon).
\end{aligned}$$
By Lemma \ref{phi est outside},
$$|E_{2}|\leqslant C\varepsilon^{n-1}\delta^{2-2n}.
$$

So together with Lemma \ref{f est} we get
$$\begin{aligned}
\left(\int_{\partial M}\phi_{2}^{\frac{2(n-1)}{n-2}}d\sigma_g\right)^{\frac{n-2}{n-1}}\leqslant  \left(\int_{\partial M}f\phi_{2}^{\frac{2(n-1)}{n-2}}d\sigma_{g}\right)^{\frac{n-2}{n-1}}+C\varepsilon^{n-1}\log(\delta/\varepsilon)+C\varepsilon^{n-1}\delta^{2-2n}.
\end{aligned}$$
From this, the assertion follows easily.
\end{proof}

We follow \cite{article3} to consider the set
$$\mathcal{Z}=\{p\in M:\limsup _{x \rightarrow p} d(p, x)^{2-d}\left|W_g\right|(x)=0\},$$
where $W_g$ denotes the Weyl tensor of $(M,g)$, and $d=[(n-2)/2]$. In other words, a point $p\in M$ belongs to $\mathcal{Z}$
if and only if $\nabla^mW_g(p)=0$ for all $m=0, 1, \cdots, d-2$. Note that the set $\mathcal{Z}$ is invariant under a conformal change of the metric. Also note here we require $n\geqslant 6$ in order for $d-2\geqslant0$.

In fact, the set $\mathcal{Z}$ can also be characterized using the algebraic Weyl tensor $Z_{ijkl}$. More precisely, the relationship between the Schouten tensor $B$ and the algebraic Schouten tensor $A$ is given by
$$
B_{i k}=R_{i k}-\frac{R_g}{2(n-1)} g_{ik}=\frac{1}{2} A_{i k}+O\left(|x|^{d-1}\right)+O(|\partial(h \partial h)|).
$$
While the relationship between the Weyl tensor $W$ and the algebraic Weyl tensor $Z$ is given by
$$
\begin{aligned}
    W_{i j k l}&=R_{i j k l}-\frac{1}{n-2}\left(B_{i k} g_{j l}-B_{i l} g_{j k}-B_{j k} g_{i l}+B_{j l} g_{i k}\right)\\
    &=-\frac{1}{2}Z_{ijkl}+O(|x|^{d-1})+O(|\partial (h\partial h)|).
\end{aligned}
$$
If $p \in \mathcal{Z}$, then by definition $\nabla^m W_g(p)=0$ for all $m \leqslant d-2$, which implies $W_{i j k l}(x) = O(|x|^{d-1})$. Substituting this into the equation above yields $Z_{ijkl}(x) = O(|x|^{d-1})$. Recall that $H_{ik}$ is defined as a polynomial of degree $d$, which means $Z_{ijkl}$, being comprised of the second derivatives of $H_{ik}$, is a polynomial of degree at most $d-2$. Since $Z_{ijkl}$ is a polynomial of degree $d-2$ that satisfies $Z_{ijkl}(x) = O(|x|^{d-1})$ as $x \to 0$, it must be identically the zero polynomial. Therefore, $p\in\mathcal{Z}$ implies that $Z_{ijkl}=0$ everywhere in $\mathbb{R}_{+}^n$.

	

\begin{proof}[Proof of Theorem \ref{umbilic}]
We will consider two different cases based on whether the point $p\in\partial M$ lies in $\mathcal{Z}$ or not.

\textit{Case I.} If $p\in\mathcal{Z}$, then $Z_{ijkl}=0$ in $\mathbb{R}_{+}^n$. As established above, the umbilic boundary geometry ensures $\partial_n H_{ik} = 0$ on $\partial \mathbb{R}_{+}^n$. Therefore, by Proposition \ref{vanishing  H}, we conclude that $H \equiv 0$ in $\mathbb{R}_{+}^n$. Consequently, $\sum_{i,k=1}^{n}\sum_{2\leqslant|\alpha|\leqslant d}|h_{ik,\alpha}|^{2}=0.$ Thus from Proposition \ref{5.1},

$$
\begin{aligned}
    E_{g}(\phi_{2}) & \leqslant \mathcal{Q}(B^{n},\partial B^{n})\left(\int_{\partial M} f\phi_{2}^{\frac{2(n-1)}{n-2}} d \sigma_g\right)^{\frac{n-2}{n-1}}\\
	&\quad - \left(\mathcal{I}(p, \delta)-C\delta^{2d+4-n}-C\varepsilon\log(\delta/\varepsilon)-C\varepsilon\delta^{2-2n}-C\varepsilon^{2}\delta^{-n}\right)\varepsilon^{n-2}.
\end{aligned}
$$

We follow \cite{article4} to consider the manifold $(M\setminus\{p\}, G^{4/(n-2)}g)$. Recall that $g = u^{4/(n-2)}g_0$ with $u(p)=1$. Thus, the Green's functions are related simply by $G = u^{-1}G_{g_0}$, where $G_{g_0}$ is the Green's function for the conformal Laplacian of $g_0$. Consequently, the blow-up metric is exactly $G^{4/(n-2)}g = G_{g_0}^{4/(n-2)}g_0$. Because the background metric $g_0$ has exactly vanishing mean curvature everywhere, this blow-up manifold is scalar flat and its boundary is exactly totally geodesic. After doubling this manifold, we obtain an asymptotically flat manifold with zero scalar curvature, denoted as $(\tilde{M}, \tilde{g})$. By \cite{article4}*{Proposition 4.3}, its ADM mass $m(\tilde{g})$ is well-defined and 
$$m(\tilde{g})=C_0\mathcal{I}(p,\delta)+O(\delta^{2d+n-4})$$ 
for some fixed $C_0>0$. Since $(M, g)$ is not conformally equivalent to the ball, $(M\setminus\{p\}, G^{4/(n-2)}g)$ is not conformally equivalent to the ball either. So by the positive mass theorem \cite{article23}, we must have $m(\tilde{g})>0$. (In particular, $m(\tilde{g})\ne0$ due to the rigidity result.)

Thus, for fixed $C_0>0$, we can take $\delta$ small enough and get
$$
\mathcal{I}(p,\delta)-C\delta^{2d+4-n}=C_0^{-1}m(\tilde{g})-C'\delta^{2d+4-n}>0.
$$
Then for fixed $\delta$, we can take $\varepsilon\ll1$ to make
$$\mathcal{I}(p, \delta)-C\delta^{2d+4-n}-C\varepsilon\log(\delta/\varepsilon)-C\varepsilon\delta^{2-2n}-C\varepsilon^{2}\delta^{-n}>0,$$
and as a result,
$$E_{g}(\phi_{2})<\mathcal{Q}(B^{n},\partial B^{n})\left(\int_{\partial M}f\phi_{2}^{\frac{2(n-1)}{n-2}} d \sigma_{g} \right)^{\frac{n-2}{n-1}}.$$\\

\textit{Case II.} If $p\notin\mathcal{Z}$, then there exists a smallest $|\alpha_{0}|\leqslant d$ with $h_{ik,\alpha_{0}}\neq 0.$ 

If $|\alpha_{0}|<\frac{n-2}{2}$, 
$$\begin{aligned}
    &\quad\,\varepsilon^{n-2}\int_{B_\delta \cap \mathbb{R}_{+}^n}(\varepsilon+|x|)^{2|\alpha_{0}|+2-2n}dx\\
    & = \varepsilon^{2|\alpha_{0}|}\int_{0}^{\delta/\varepsilon}\frac{r^{n-1}}{(1+r)^{2n-2-2|\alpha_{0}|}}dr\\
    & > \varepsilon^{2|\alpha_{0}|}\int_{0}^{1}\frac{r^{n-1}}{(1+r)^{2n-2-2|\alpha_{0}|}}dr\\
   & = C\varepsilon^{2|\alpha_{0}|}.
\end{aligned}$$
So by Proposition \ref{5.1},
$$
\begin{aligned}
    E_{g}(\phi_{2}) & <   \mathcal{Q}(B^{n},\partial B^{n})\left(\int_{\partial M} f\phi_{2}^{\frac{2(n-1)}{n-2}} d \sigma_g\right)^{\frac{n-2}{n-1}}\\
	&\quad -\left(C\varepsilon^{2|\alpha_{0}|-(n-2)}-C'\varepsilon\log(\delta/\varepsilon)-C_{1}(\delta)-C_{2}(\delta)\varepsilon-C_3(\delta)\varepsilon^{2}\right)\varepsilon^{n-2}.
\end{aligned}
$$
Note that $2|\alpha_{0}|-(n-2)<0$, so for fixed $\delta$, we can take $\varepsilon\ll1$ to make
$$
C\varepsilon^{2|\alpha_{0}|-(n-2)}-C'\varepsilon\log(\delta/\varepsilon)-C_{1}(\delta)-C_{2}(\delta)\varepsilon-C_3(\delta)\varepsilon^{2}>0,
$$
and as a result,
$$E_{g}(\phi_{2})<\mathcal{Q}(B^{n},\partial B^{n})\left(\int_{\partial M}f\phi_{2}^{\frac{2(n-1)}{n-2}} d \sigma_{g} \right)^{\frac{n-2}{n-1}}.$$

If $|\alpha_{0}|=\frac{n-2}{2}$, then $d=\frac{n-2}{2}.$
$$ 
\begin{aligned}
    &\quad\,\varepsilon^{n-2}\int_{B_\delta \cap \mathbb{R}_{+}^n}(\varepsilon+|x|)^{2|\alpha_{0}|+2-2n}dx\\
    & =\varepsilon^{n-2}\int_{0}^{\delta/\varepsilon}\frac{r^{n-1}}{(1+r)^{n}}dr\\
    & > \varepsilon^{n-2}\int_{1}^{\delta/\varepsilon}\frac{r^{n-1}}{(1+r)^{n}}dr\\
    & >C\varepsilon^{n-2}\log(\delta/\varepsilon).
\end{aligned}
$$
So by Proposition \ref{5.1},
$$
\begin{aligned}
    E_{g}(\phi_{2}) & <   \mathcal{Q}(B^{n},\partial B^{n})\left(\int_{\partial M} f\phi_{2}^{\frac{2(n-1)}{n-2}} d \sigma_g\right)^{\frac{n-2}{n-1}}\\
	&\quad -\left(C\log(\delta/\varepsilon)-C'\varepsilon\log(\delta/\varepsilon)-C_{1}(\delta)-C_{2}(\delta)\varepsilon-C_3(\delta)\varepsilon^{2}\right)\varepsilon^{n-2}.
\end{aligned}
$$
For fixed $\delta$, we can take $\varepsilon\ll1$ to make
$$
C\log(\delta/\varepsilon)-C'\varepsilon\log(\delta/\varepsilon)-C_{1}(\delta)-C_{2}(\delta)\varepsilon-C_3(\delta)\varepsilon^{2}>0,
$$
and as a result,
$$E_{g}(\phi_{2})<\mathcal{Q}(B^{n},\partial B^{n})\left(\int_{\partial M}f\phi_{2}^{\frac{2(n-1)}{n-2}} d \sigma_{g} \right)^{\frac{n-2}{n-1}}.$$
\end{proof}

\begin{rmk}
In fact, the assumption that ``$M$ is not conformally equivalent to the ball" can be slightly weakened. By \cite{article4}*{Proposition 4.3}, the limit $\lim_{\delta\rightarrow0}\mathcal{I}(p, \delta)$ exists. Therefore, the only requirement for the proof above is
$$
\lim_{\delta\rightarrow0}\mathcal{I}(p, \delta) > 0.
$$
\end{rmk}



\section*{Acknowledgments}
The first author is deeply grateful to Professor Weimin Sheng for his initial guidance and for introducing him to this field. The second author sincerely thanks Professor Rick Schoen for his invaluable insights through many fruitful discussions and for his constant encouragement.

\bibliographystyle{amsplain}

\end{document}